\documentclass[11pt, a4paper]{amsart}
\usepackage{amsfonts,amsmath, amsthm, amssymb, amscd}
\input{amssym.def}

\usepackage[dvips]{graphicx}
\usepackage{epsfig}
\usepackage{psfrag}

\newcommand{\tri}{\;\;\makebox[0pt]{$|$}\makebox[0pt]{$\cap$}\;\;}

\theoremstyle{plain}
\newtheorem{theorem}{Theorem}[section]
\newtheorem{lemma}[theorem]{Lemma}
\newtheorem{cor}[theorem]{Corollary}
\newtheorem{prop}[theorem]{Proposition}
\newtheorem{dfn}[theorem]{{Definition}}
\theoremstyle{definition}
\newtheorem*{remo}{Remark}
\newtheorem*{remos}{Remarks}
\newtheorem*{prob}{Open Problems}

\numberwithin{equation}{section}

\newcommand {\Bz}{\mathbb{Z}}
\newcommand {\Br}{\mathbb{R}}
\newcommand {\Bq}{\mathbb{Q}}

\begin{document}

\title[Minimal Geodesic Foliation on $T^2$]
{Minimal Geodesic Foliation on $T^2$ in case of vanishing topological entropy}

\author{Eva Glasmachers and Gerhard Knieper }
\date{\today}
\address{Faculty of Mathematics,
Ruhr University Bochum, 44780 Bochum, Germany}
\email{eva.glasmachers@rub.de, gerhard.knieper@rub.de}



\begin{abstract}
On a Riemannian 2-torus $(T^2,g)$  we study the geodesic flow in the case
of low complexity described by zero topological entropy.
We show that this assumption implies a nearly integrable behavior.
In our previous paper \cite{GK} we already obtained that the asymptotic
direction and therefore also the rotation number exists for all geodesics.
In this paper we show that
for all $r \in \mathbb{R} \cup \{\infty\}$ the universal cover $\Br^2$ is foliated
by minimal geodesics of rotation number $r$. For irrational $r \in \mathbb{R}$ all
geodesics are minimal, for rational $r \in \mathbb{R} \cup \{\infty\}$ all geodesics
stay in strips between neighboring minimal axes. In such a strip the minimal geodesics
are asymptotic to the neighboring minimal axes and generate two foliations.
\end{abstract}


\maketitle


\section{Introduction}

This paper continues our work \cite{GK}  on geodesic flows on the
unit tangent bundle of a two-dimensional Riemannian torus
$(T^2,g)$. We also like to mention that earlier versions of the
results obtained in \cite{GK} are already contained in the thesis
of the first author \cite{G}.

The goal of our work is to study dynamical and geometrical
implications of vanishing topological entropy. Recall that the
topological entropy of a continuous dynamical system represents
the exponential growth rate of orbit segments distinguishable with
arbitrarily fine but finite precision. It therefore describes the
total exponential orbit complexity with a single number. Due to a
theorem of A.~Katok~\cite{K80} positive topological entropy and
the existence of a horseshoe are equivalent provided the phase
space of the flow is 3-dimensional.

It turns out that zero topological entropy yields strong
restrictions on the behavior of geodesics. Important results in
this direction are due to J.~Denvir and R.~S.~MacKay~\cite{DM}.
Their work implies that contractible closed geodesics on $T^2$ do
not exist in case of vanishing topological entropy. An independent
proof was also given in \cite{G}. Using variational methods
S.~V.~Bolotin and P.~H.~Rabinowitz~\cite{BR} studied the
complexity of the geodesic flow on $T^2$ and obtained positive
topological entropy under certain conditions.

In  \cite{GK} we showed that absence of positive topological
entropy implies nearly integrable behavior. In particular, the
lifts of all geodesics on $\Br^2$ (not just the minimal ones) stay
in tubular neighborhoods of Euclidean lines. Hence, all geodesics
have an asymptotic direction and define a nontrivial continuous
constant of motion. The main tools used in our approach are
curve-shortening techniques that allow to globally control
geodesics in the presence of certain intersection patterns of
geodesic segments.

In this paper we strengthen our previous results. We show that
each asymptotic direction yields a geodesic foliation on the
Riemannian universal covering $(\mathbb{R}^2, \tilde{g})$ of $(T^2
= \Br^2 / \Bz^2 ,g)$ consisting of globally minimizing geodesics.
For irrational directions the foliation is unique and all
geodesics with irrational directions are minimal. In particular,
each geodesic $c$ on $\mathbb{R}^2$ with irrational rotation
number does not intersect its translates, i.e. $\tau (c) \cap c =
\emptyset$ for all $\tau \in \Bz^2\setminus \{(0,0)\}$. However,
for all rational directions the foliations are unique if and only
if the metric is flat. We remark that for monotone twist maps
analogous result were obtained by S.~Angenent~\cite{A-1992}.
Hence, our results extend well known relations between minimal
orbits of monotone twist-maps (Aubry-Mather theory
\cite{Aubry},\cite{Mather}) and minimal geodesics on $T^2$ which
were studied by  V.~Bangert~\cite{B1988} as well as M.~L.~Bialy
and L.~V.~Polterovich~\cite{BP}. However, our approach does not
use monotone twist maps but again relies on the curve shortening
flow.

\noindent Our main results can be summarized in the following theorem:\\

\noindent{\bf Theorem.}
{\it Let $(T^2,g)$ be a Riemannian torus with $h_{top}(g) =0$. Then for
all
$r \in \Br \cup \{\infty\}$ the torus is foliated by minimal geodesics
with rotation number $r$. If $r$ is irrational the foliation is unique.
For all $r \in \Bq \cup \{\infty\}$ the foliations are unique iff $T^2$
is flat. Moreover, each lift of a geodesic on $\mathbb{R}^2$ with irrational
rotation number does not intersect their translates.}

\begin{remos}
\begin{enumerate}
\item[(a)]
We remark that each such minimal foliation on $T^2$ corresponds to
the graph of a Lagrangian torus on the unit tangent bundle $ST^2$
invariant under the geodesic flow. Therefore, a result of I.~V.~
Polterovich~\cite{Polt} implies that each irrational minimal
geodesic is dense in $T^2$. In particular, he shows that the
metric is flat provided one irrational minimal geodesic has no
focal points.
\item[(b)] Further analogies to properties of orbits of monotone
twist maps in the case of vanishing topological entropy presented
in \cite{A-int} were derived in \cite{GK} and will be summarized
in this paper in Theorem~\ref{summarize}.
\end{enumerate}
\end{remos}

\section{Topological entropy and properties of minimal geodesics}

A fundamental concept in our investigation is the topological
entropy of a continuous dynamical system. It is invariant under
topological conjugations and measures, as described in the
introduction, the exponential orbit complexity with a single
non-negative number. The precise meaning becomes apparent in the
following definition introduced by R.~E.~Bowen~\cite{BO}.

\begin{dfn}[Topological Entropy] \label{topent}
Let $(Y,d)$ be a compact metric space, $\phi^t: Y \to Y$ a continuous
flow and
$d_{T}(\cdot,\cdot)$ the dynamical metric defined by
$d_{T}(v,w): = \max_{0 \leq t \leq T}d(\phi^tv, \phi^tw)$ for all $v,w
\in Y$.
For a given $\varepsilon>0$ a subset $F\subset Y$ is called
$(\phi,\varepsilon,d_{T} )$-{\it separated}
set of $Y$, if for $x_1 \not =x_2 \in F$ we have
$d_{T}(x_1,x_2)> \varepsilon.$
The topological entropy of $\phi^t$ is defined as
$$
h_{top}(g)  = h_{top}(\phi) =\lim_{\varepsilon \to 0} \limsup_{T \to
\infty}
\left(\frac{1}{T} \log r_T(\phi, \varepsilon)  \right),
$$
where $r_T(\phi, \varepsilon)$ is the maximum of the cardinalities
of any
$(\phi,\varepsilon, d_T)$-separated set of $Y$.
\end{dfn}
For more details and properties of the topological entropy see for
example
\cite{KH} or \cite{Wo}.
In the investigation of the dynamics of the geodesic flow on $T^2 = \Br^2 / \Bz^2$ the
notion of rotation number is of central importance.

\begin{dfn}
We say that a geodesic $c:\Br \to \Br^2$ has an asymptotic direction
if the limit
$$\delta(c) :=\lim\limits_{t\to \infty} \frac{c(t)}{\|c(t)\|} \in S^1$$
exists.
If $\delta(c) =(x,y)$ we call the projection onto $\mathbb{P}_1(\Br)$ given by
$$
 \rho(c) = \left\{ \begin{array}{cc}
 \frac{y}{x}, & \text{if} \; x \not=0\\
 \infty, & \text{otherwise}
 \end{array} \right.
 $$
the rotation number of $c$.
\end{dfn}

\begin{dfn}
A geodesic $c: \Br \to \Br^2$ on the Riemannian universal covering
$(\Br^2, \tilde g)$
of $(T^2, g)$ is called an axis if
there exists a nontrivial translation element $\tau \in  \Bz^2$ such that $\tau c(t) = c(t + l)$
for some $l \in \Br$ and all $t \in \Br$.
\end{dfn}
\begin{remo}
The projection of an axis $c: \Br \to \Br^2$ of a nontrivial translation element
$\tau \in  \Bz^2$ onto $T^2$ corresponds to a closed geodesics in the homology class given by $\tau$.
\end{remo}

For surfaces of genus strictly larger than one Morse \cite{Morse}
began in 1924 a systematic investigation of minimal geodesics,
i.e., geodesics which lift to minimal geodesics on the universal
covering. Somewhat later Hedlund \cite{Hedlund} obtained similar
results in the case of the 2-torus.

Minimal geodesics on $(\Br^2, \tilde g ) $ and their projections on
$(T^2, g)$ will play an important role in this paper. Central properties
are that two different minimal geodesics on $(\Br^2,\tilde g)$ cross at most once
and minimal geodesics have no self-intersections.

There are well known connections between minimal geodesics  on
$(\Br^2,\tilde g)$ and rotation numbers: A fundamental result of
Hedlund \cite{Hedlund} yields that for each $r \in \Br \cup
\{\infty\}$ there exists a minimal geodesic $c:\Br \to \Br^2$ with
rotation number $r$. Furthermore, there exists a constant $D>0$,
such that to each minimal geodesic $c: \Br \to \Br^2$ corresponds
a Euclidean line $l$, and to each Euclidean line $l$ corresponds a
minimal geodesic $c$ such that
$$d (l, c(t)) \leq D, \quad \text{for all $t \in \Br.$}$$
In particular, this implies the existence of the rotation number
for each minimal geodesic.

The set of minimal geodesics with a fixed irrational rotation
number is totally ordered, i.e., in this set no pair of geodesics
intersects. In the set of minimal geodesics with a fixed rational
rotation number the subset of minimal axes is totally ordered as
well. Two minimal axes $\alpha_1, \alpha_2$ with the same
asymptotic directions bounding a strip containing no further
minimal axes are called neighboring minimals. There exists a
minimal geodesic $c: \Br \to \Br^2$ of asymptotic type
$A(\alpha_1, \alpha_2)$, i.e., $d(c(t), \alpha_1(\Br)) \to 0$ for
$t \to \infty$ and $d(c(t), \alpha_2(\Br)) \to 0$ for $t \to
-\infty$. If $\alpha_1 \not= \alpha_2$, then each pair $c_1$ and
$c_2$ of minimal geodesics of asymptotic type $A(\alpha_1,
\alpha_2)$ and $ A(\alpha_2, \alpha_1)$, respectively, has a
unique intersection. Moreover, all minimal geodesics with the same
rational rotation number and the same asymptotic type are totally
ordered on $\Br^2$. For more details see \cite{B1988}.\\

In a recent paper \cite{GK} we derived under the assumption of vanishing
topological entropy strong properties for all geodesics on the universal
covering. Early versions of the results where already obtained in \cite{G}.
The main properties are summarized in the following theorem.

\begin{theorem} [see \cite{GK}] \label{summarize}
Let $g$ be a Riemannian metric on $T^2$ with vanishing topological
entropy. Then, on the Riemannian universal covering $( \Br^2, \tilde{g})$, every geodesic
$c$ is escaping, i.e., $\lim\limits_{t \to \pm \infty}\|c(t)\| =\infty$, has no
self-intersections and for every geodesic the asymptotic direction $\delta(c)$ exists
with the additional property that $\delta(c)=-\delta(c^-)$ if $c^-(t) =c(-t)$ is the
geodesic traversed in the opposite direction. Furthermore, each geodesic $c$ lies in a
strip in $\Br^2$ bounded by two parallel Euclidean lines. \\
Moreover, the asymptotic direction defines a continuous function $\delta: S\mathbb{R}^2 \to S^1$
given by $\delta(v) =\delta(c_v)$ such that
for each $x \in \mathbb{R}^2$ the restriction  $\delta_x: S_x\mathbb{R}^2\to S^1$
to the fibers of $ \mathbb{R}^2$ is surjective.
\end{theorem}

\begin{remo}
Theorem \ref{summarize} implies   $\delta(v) = - \delta(-v)$ and $\rho(v)=\rho(-v)$
for all $v \in ST^2$. Furthermore, $\delta$ and hence $\rho$ induces a flow invariant
continuous function on $ST^2$.
\end{remo}

The main technical ingredient used in the proof of this Theorem is the following Fundamental Lemma
derived in \cite{GK}. It will be crucial in this paper as well.

\begin{lemma}[Fundamental Lemma]
Let $g$ be a Riemannian metric on $T^2$ and $\alpha: \Br \to \Br^2$ a
minimal axis of the translation element $\tau$. Let $c_1:[0, a] \to \Br^2$ and
$c_2:[0, b] \to \Br^2$ be two geodesic segments with endpoints on $\alpha$ and
$$
c_1((0, a)) \cap \alpha(\Br)=\emptyset,  \; c_2((0, b)) \cap \alpha(\Br)=\emptyset.
$$
Assume that there exists a translation element $\eta$, with
$\eta\alpha(\Br) \cap \alpha(\Br)= \emptyset$ such that
$$
\eta \alpha(\Br) \tri c_1([0, a]) \not= \emptyset \; \; \text{and}
\; \; \eta^{-1} \alpha(\Br) \tri c_2([0, b]) \not= \emptyset.
$$
Then the metric $g$ has
positive topological entropy.
\end{lemma}

The proof of this lemma heavily relies on the curve shortening flow,
see \cite{GO}.
As shown in the proof of Theorem~II in
\cite{GK} the Fundamental Lemma generalizes also to broken geodesic segments $c_1,c_2$
such that the exterior angles in the singularities of $c_1,c_2$ are smaller than $\pi$.
The exterior angles of $c_1$  are the interior angles in the unbounded connected component of
$H_1 \setminus c_1([0,a])$ where $H_1$ is the halfplane in $\Br^2\setminus \alpha(\Br)$
containing $c_1(0,a)$.

\section{Structure of the geodesics in case of zero topological entropy}

\begin{lemma} \label{1}
Let $(T^2,g)$ be a Riemannian torus with vanishing topological entropy
and $(\Br^2, \tilde{g})$ the Riemannian universal cover with the lifted metric $\tilde{g}$.
Then any pair of geodesics $c_1, c_2 : \Br \to \Br^2$
with $\rho(c_1)
\not = \rho(c_2)$ has at most one intersection.
\end{lemma}

\begin{proof}
According to Theorem \ref{summarize} no geodesic on the
universal covering  has self-intersections provided the topological entropy is zero.
Assume that there exists a pair $c_1, c_2: \Br \to \Br^2$ of geodesics with two intersections
and such that $\rho(c_1) \not = \rho(c_2)$. After reparameterization and change of orientation
of the geodesics we may assume: $c_1(0)=c_2(0)$ and  $c_1(t_1)=c_2(t_2)$ for $t_1,t_2 >0$.
Since $\rho(c_2) \not = \rho(c_1)$
 and $\delta(c_2) = - \delta(c_2^{-})$, there also exist times $t_3,t_4 <0$ such
that $c_1(t_3) =c_2(t_4)$. Consider the piecewise geodesic curves
we get by gluing $c_1([0, \infty))$ with $c_2((-\infty, 0])$ and
analogously gluing $c_2([0, \infty))$ with $c_1((-\infty, 0])$.

\begin{figure}[h!]
\begin{center}
\psfrag{a}{$c_1$}
\psfrag{b}{$c_2$}
\psfrag{v}[c]{\small $c_1(- \infty,0]$}
\psfrag{w}{\small $c_2[0, \infty)$}
\psfrag{x}[c]{\small $c_2(- \infty, 0]$}
\psfrag{y}{\small $c_1[0, \infty)$}
\psfrag{e}{$\alpha$}
\psfrag{f}{$\eta\alpha$}
\psfrag{g}{$\eta^{-1} \alpha$}
\includegraphics[scale=0.25]{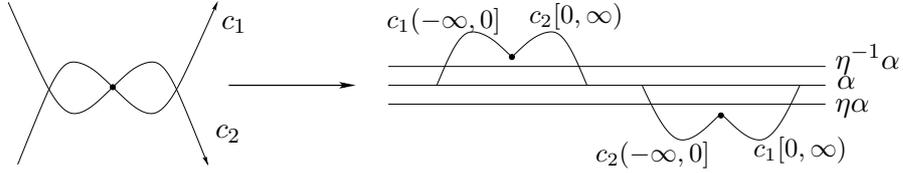}
\end{center}
\caption{$V$-shaped broken geodesic segments in the proof of Lemma~\ref{1}.}
\end{figure}
There exists a minimal axis $\alpha$ and a translation
element $\eta$
such that subsegments of these $V$-shaped broken geodesic curves by
construction fulfill the assumption
of the Fundamental Lemma in \cite{GK}. Hence, we conclude positive topological
entropy in
contradiction to the assumption.
\end{proof}
For $x \in \Br^2$ we define a lift $\tilde \delta_x:\Br \to \Br$ of the asymptotic
direction $\delta_x: S_x\mathbb{R}^2\to S^1$.
Let $e_1 = (1,0)$ and $e_2 = (0,1)$ be the standard orthonormal basis in $T_x\Br^2 \cong \Br^2$.
Choose an orthonormal basis $v_1, v_2 \in T_x\Br^2$ with
respect to the metric $g_x$ and the same orientation as $e_1,e_2$. Consider the
coverings $p_1 : \Br \to  S_x\mathbb{R}^2$ and $p_2 :\Br \to S^1$ given by
$p_1(t) = \cos t v_1 + \sin t v_2$ and  $p_2(t) = \cos t e_1 + \sin t e_2$.
Since $\delta_x: S_x\mathbb{R}^2\to S^1$ is continuous there exists a lift
$\tilde \delta_x:\Br \to \Br$, unique up to a multiple of $2 \pi$, defined by
$$
p_2(\tilde \delta_x(t)) = \delta_x(p_1(t)) \enspace.
$$

\begin{lemma}\label{2}
Let $(T^2,g)$ be a Riemannian torus with zero topological entropy.
Then, for all $x \in T^2$ the lift  $\tilde \delta_x:\Br \to \Br$  of the asymptotic
direction $\delta_x: S_xT^2 \to S^1$ is a monotone function. Moreover,
$ \tilde \delta_x(t + 2 \pi) -\tilde \delta_x(t) = 2 \pi$ and therefore the degree
of  $\delta_x$ is one.
\end{lemma}

\begin{proof}
Choose $0 \le t_1 < t_2 < \pi$.  Then $w_1 = p_1(t_1)$, $w_2 = p_1(t_2)$ are positively
oriented and $c_{w_1}(\Br)$ and $c_{w_2}(\Br)$ are contained in Euclidean strips
corresponding to the directions $\delta_x(w_1)$ and $\delta_x(w_2)$, respectively.
If $\delta_x(w_1) \not= \delta_x(w_2)$ the ray $c_{w_2}((0, \infty)$ is by
Lemma ~\ref{1} contained in a single connected component
of $\Br^2 \setminus c_{w_1}(\Br)$. This implies that
$\delta_x(w_1)= p_2(\tilde \delta_x(t_1))$, $\delta_x(w_2)=p_2(\tilde \delta_x(t_2))$
have the same orientation as $w_1, w_2$ and therefore are positively oriented as well.
Hence, $ 0 <\tilde \delta_x(t_2)-\tilde \delta_x(t_1)< \pi$. Furthermore,
$\tilde \delta_x(t +\pi)-\tilde \delta_x(t) = \pi$ for all $t$ which yields the second assertion.
\end{proof}

\noindent Lemma~\ref{2} implies the following Corollary:

\begin{cor} \label{c1}
Let $(T^2,g)$ be a Riemannian torus with zero topological entropy and $(\Br^2, \tilde{g})$
the Riemannian universal cover. Then for all $x \in
\Br^2$, $S_x \Br^2$ is a disjoint union of the closed sets
$$S_x^{r} = \{v \in S_x \Br^2 \mid \rho(c_v) =r\} \enspace,$$
where $c_v: \Br \to \Br^2$ is a geodesic with $\dot c_v(0)=v$ and $r \in \Br \cup
\{\infty\}$.
For each $r$ the set $S_x^{r}$ consists of two connected closed
antipodal components.
\end{cor}

In the sequel we will need the following theorem which summarizes important results on
minimal rays obtained by Bangert \cite{B1994}.

\begin{theorem} \label{B}
Let $(T^2,g)$ be a Riemannian torus and $(\Br^2, \tilde{g})$ the Riemannian universal cover.
\begin{enumerate}
\item
 Let $c_1: [0, \infty) \to \mathbb{R}^2 $ and $c_2: [0, \infty) \to \mathbb{R}^2 $ two minimal rays
with $c_1(0) = c_2(0)$ and $\dot{c}_1(0) \not= \pm \dot{c}_2(0) $.
Then, for each $\epsilon >0$ the
rays $c_1: [-\epsilon, \infty) \to \mathbb{R}^2 $ and  $c_2: [-\epsilon, \infty)\to \mathbb{R}^2 $
are not minimal, provided $c_1,c_2$ are asymptotic or $c_1,c_2$ have the same irrational asymptotic direction.
\item
For any minimal ray $c:[0, \infty) \to \mathbb{R}^2$ with a
rational rotation number $r$ there exists a unique pair of neighboring minimal axes
$ \alpha_1, \alpha_2$ with rotation number $r$ and the following properties:
\begin{enumerate}
\item
$c[0, \infty)$ is contained in the strip bounded by  $ \alpha_1$ and $ \alpha_2$.
\item
$c:[0, \infty) \to \mathbb{R}^2$ is asymptotic to $ \alpha_1$ or $ \alpha_2$.
\end{enumerate}
\item
For  each $x$ between a
pair of neighboring minimal axes $\alpha_1, \alpha_2$ with the same asymptotic direction there exist four
different asymptotic
minimal geodesic rays $c_1^+, c_1^-, c_2^+$ and $c_2^-$  with $c_i^{\pm}(0) = x$ for $i\in \{1,2\}$ and
$$
\lim\limits_{t \to \infty} \|c_i^{\pm}(t) - \alpha_i(\pm t+ s_k)\|=0
$$ for some constants $s_k$ with $k \in \{1,2,3,4\}$.
\end{enumerate}
\end{theorem}
\begin{proof}

Under the assumption that $c_1$ and $c_2$ are asymptotic, assertion (1) is an easy consequence of the triangle
inequality.
If $c_1$ and $c_2$ have the same irrational asymptotic direction assertion (1)
follows from Theorem~3.6 in \cite{B1994}.\\
The assertions (2) and (3) follow from Theorem~3.7 in \cite{B1994}.
\end{proof}

\begin{lemma} \label{3}
Let $(T^2,g)$ be a Riemannian torus with zero topological entropy and  $(\Br^2, \tilde{g})$ the
Riemannian universal cover. Let $x \in \Br^2$ be fixed.\\
Then for all $r \in \Br \cup \{\infty\}$ and each $v
\in \partial S_x^r$ any geodesic $c_w:\Br \to \Br^2$ with $w \in S_x \Br^2 \setminus \{v \}$ intersects
$c_v$ only in $x$.
In particular, $c_v:[0, \infty) \to \Br^2$ and $c_{-v}:[0, \infty) \to \Br^2$ are minimal rays.\\
If $r$ is rational there exists a pair of minimal neighboring axes $\alpha_1$, $\alpha_2$
with the following properties:
Each of the minimal rays $c_v$ corresponding to $v \in \partial S_x^r$
is asymptotic to $\alpha_1$ or $\alpha_2$, or it coincides with one of the axes.
\end{lemma}

\begin{proof}
For  $r \in \Br \cup \{\infty\}$ and $x \in \Br^2$ consider $v \in \partial S_x^{r}$ and
$w \in S_x \Br^2$ with $w \not= v$.
Assume that $c_w$ intersects $c_{v}( \Br) \setminus \{x\}$. According to Lemma ~\ref{1}
this implies $\rho(w) = \rho(v)$.
Choose a sequence
$v_n \in  S_x \Br^2 \setminus S_x^{r}$ converging to $v$.
By the continuous dependence of geodesics on initial conditions
$c_{v_n}( \Br) \setminus \{x \}$ intersects $c_w$ for sufficiently  large $n$ as well.
Since $\rho(v_n) \not= \rho(w)$ this contradicts Lemma ~\ref{1}. In particular, both rays
$c_v: [0, \infty) \to \Br^2$ and $c_{-v}: [0, \infty) \to \Br^2$ are minimal.
If $r$ is rational and $v \in \partial S_x^{r}$, then the minimality of $c_v:[0, \infty) \to \Br^2$
and Theorem ~\ref{B} imply that $c_v$ is asymptotic to one of the two neighboring minimal
axes with rotation number $r$.
\end{proof}

\begin{prop}\label{4}
Let $(T^2,g)$ be a Riemannian torus with vanishing topological entropy. Then the set
$\{\partial S_x^r \mid x \in \mathbb{R}^2\}$ is invariant under the geodesic flow.
If $r$ is irrational, then each $S_x^{r}$
consists of a pair of antipodal vectors $\{ \pm v \}$.
\end{prop}

\begin{proof}
Consider $v \in \partial S_x^r$ and choose $t > 0$. Since $\rho$ is flow invariant we have
$\phi^t(v) =:v_t \in S_{c_v(t)}^r$.
We claim that $v_t \in \partial S_{c_v(t)}^r$. Assume first that $r$ is irrational.
If $v_t \not \in \partial S_{c_v(t)}^r$, choose $w \in  \partial S_{c_v(t)}^r$.
But this contradicts Theorem \ref{B} since $c_{v_t}: [-t, \infty) \to \mathbb{R}^2$ is minimal.
Assume  that $r$ is rational. Then $c_v$ is forward asymptotic to a minimal axis $\alpha$.
If $v_t \not \in \partial S_{c_v(t)}^r$, there exists $w \in  \partial S_{c_v(t)}^r$ such that
$c_w$ is also asymptotic to $\alpha$. But then $c_{v_t}$ is asymptotic to $ c_w$ which
contradicts Theorem \ref{B}. Hence, we obtain $v_t=w$.\\
Since $-v \in  \partial S_x^r$ we have $\phi^t(-v) \in \partial S_{c_{-v}(t)}^r =  \partial S_{c_{v}(-t)}^r$
for all $t >0$. But this implies $\phi^{-t}(v) = -\phi^t(-v) \in \partial S_{c_{v}(-t)}^r$
and therefore the flow invariance of $\{\partial S_x^r \mid x \in \mathbb{R}^2\}$
for negative times.
\end{proof}

The next Theorem gives a complete characterization
of minimal geodesics in the case of vanishing topological entropy:

\begin{theorem} \label{5}
Let $(T^2,g)$ be a Riemannian torus with zero topological entropy and $(\Br^2, \tilde{g})$
the Riemannian universal cover. Then, for all $x \in \Br^2$ it follows:
\begin{enumerate}
\item
For all $r \in \Br \cup \{\infty\}$ the
geodesics $c_v$ with $v \in \partial S_x^r$ are minimal.
\item
For all $r \in \Br \setminus \Bq$ we have $S_x^r =\partial
S_x^r =\{\pm v\}$. Hence, for each $r \in \Br
\setminus \Bq$ there exists an up to orientation and parametrization unique geodesic
with rotation number $r$ passing through $x$.
Furthermore, this geodesic is minimal.
\item
For $r \in \Bq \cup \{\infty\}$ consider the pair of neighboring minimal axes
$\alpha_1, \alpha_2$ with rotation number $r$ bounding a strip $S$ containing $x$.
\begin{enumerate}
\item
If $\alpha_1= \alpha_2$  we have $S_x^r =\partial S_x^r =\{\pm v\}$. Hence, there exists an
up to orientation and parametrization unique geodesic with rotation number $r$ passing through $x$.
Furthermore, this geodesic is a minimal axis.
\item
If $\alpha_1 \not= \alpha_2$ the set $S_x^r$ is a disjoint union of two antipodal connected sets bounded by
two minimal geodesics, where each geodesic is forward and backward asymptotic to the pair of
neighboring axes $\alpha_1, \alpha_2$. The interior of $S_x^r$
contains no minimal rays.\\
Moreover, for
all points $y$ in the interior of the strip $S$, the interior of the set $S_y^r$ is non-empty.
\end{enumerate}
\end{enumerate}
\end{theorem}

\begin{proof}
\begin{enumerate}
\item
Consider $v \in \partial S_x^{r}$ and $t>0$.  Proposition \ref{4} implies
$\phi^{-t}(v) \in \partial S_{c_{v}(-t)}^r$ and hence by Lemma \ref{3}
the geodesic $c_v :[-t, \infty) \to \Br^2$ is minimal. Since $t$ is arbitrary
$c_v: \Br \to \Br^2$ is minimal.
\item
For $r \in \Br \setminus \Bq$ consider $v, w \in \partial S_x^r$ and $v \not= -w$.
Using (1) the geodesics $c_v$ and $c_w$ are minimal which
by Theorem \ref{B} implies $v=w$.
\item
Consider $r \in \Bq \cup \{\infty \}$ and $v \in \partial S_x^r$. Then $c_v$ is asymptotic to one in the
pair of neighboring minimal axes $\alpha_1, \alpha_2$ with rotation number $r$.
\begin{enumerate}
\item
If $\alpha_1=  \alpha_2$ , the geodesic $c_v$
coincides with the axis up to orientation and $\partial S_x^r = \{ \pm v \}$.
\item
If $\alpha_1 \not=  \alpha_2$ and $c_v $ is asymptotic to say $\alpha_1$ there is
$w \in \partial S_x^r$ with $w \not= \{\pm v\}$ such that $c_w $ is asymptotic to $\alpha_2$.\\
Let  $c_u$ be  another minimal geodesic ray  with $u \in S_x^r\setminus \partial S_x^r$.
Then by Theorem~\ref{B} (2) $c_u$ is asymptotic to one of the minimal neighboring axes
and hence to a geodesic $c_v $ with  $v \in \partial S_x^r$.
Since $c_v $ is minimal this contradicts Theorem \ref{B} (1).
The last assertion follows from Corollary \ref{c1} and Theorem \ref{B} (3).
\end{enumerate}
\end{enumerate}
\end{proof}

\noindent Now we are able to prove our main theorem stated in the introduction:

\begin{proof}[Proof of the Theorem]
By Theorem~\ref{5} for $r \in \Br \setminus \Bq$ the universal covering is foliated
by minimal geodesics with rotation number $r$. Since the foliation is unique it is preserved by translation elements.
Hence, geodesics with irrational rotation numbers do not intersect their translates.
For $r \in \Bq \cup \{\infty\}$ we distinguish the following two cases:
\begin{enumerate}
\item
either the universal covering is foliated by minimal axes with rotation number $r$,
\item
or between a pair of neighboring axes $\alpha_1, \alpha_2$ with rotation number $r$ there exist two
different foliations by minimal geodesics asymptotic to $\alpha_1$ and $\alpha_2$ with rotation number $r$.
\end{enumerate}
If for all rational rotation numbers the foliation is unique it consists of minimal axes.
Then all geodesics are minimal and there exist no conjugate points. Then, by Hopf's
Theorem ~\cite{H} the Riemannian two-torus is flat.
\end{proof}
\begin{remo}
Surprisingly the only known metrics  on $T^2 = \mathbb{R}^2/ \mathbb{Z}^2$ with zero topological entropy are
the so called Liouville metrics which are of the form
$$
ds^2 = (f(x) + g(y)) (dx^2 + dy^2)
$$
where $f,g: \mathbb{R} \to \mathbb{R}$ are strictly positive smooth $1$-periodic functions.
The geodesic flow of such a metric is integrable (see e.g. \cite{BF}).  If $f,g$ are not constant the only non unique minimal geodesic foliations correspond to the directions
$(1,0)$ and $(0,1)$. If only $f$ or $g$ is constant there is exactly one such direction. If both
are constant no such direction exists since the metric is flat (see \cite{Sch}).
\end{remo}
We like to close with some important and intriguing open problems.
\begin{prob} $\;$
\begin{enumerate}
\item
Is the converse of our main theorem true, i.e. does the following hold?  Given a Riemannian metric on $T^2$ such that
for each rotation number $r \in \mathbb{R} \cup \{\infty \}$ there exists a foliation of lifted minimal geodesics
on $\mathbb{R}^2$ with rotation number $r$. Does this imply that the topological entropy is zero?
\item
Has each metric with zero topological entropy on $T^2$ only finitely many rational directions with non unique
minimal geodesic foliations? Is there at least one rational direction where the foliation is unique?
\item
Is in case of zero topological entropy the geodesic flow of a Riemannian metric on $T^2$  integrable in the sense
of Liouville and Arnold?
\item
Finally we like to add the following longstanding open question which was raised by Kozlov, Fomenko, Sinai and others
(see e.g. \cite{BF}).
Do there exist besides the Liouville metric other metrics on $T^2$ with integrable geodesic flows?
It has been conjectured by Fomenko and Kozlov (see \cite{BF}) that the answer is no.
\end{enumerate}
\end{prob}


\begin{thebibliography}{99}
\bibitem{A-int}
S.~B.~Angenent, \textit{Monotone recurrence relations, their Birkhoff orbits and
topological entropy}, Ergod. Th. and Dynam. Sys., \textbf{10} (1990), 15-41.
\bibitem{A-1992}
S.~B.~Angenent, \textit{A remark on the topological entropy and
invariant circles of an area preserving twistmap}, In 'Twist
mappings and their Applications' IMA Volumes in Mathematics,
\textbf{44}, Eds R.Mc~Gehee K.~Meyer, Springer (1992), 1-5.
\bibitem{Aubry}
S.~Aubry, P.~Y.~Le~Daeron, \textit{The discrete Frenkel-Kontorova model and its
generalizations}, Physica, \textbf{8D} (1983), 381-422.
\bibitem{B1988}
V.~Bangert, \textit{Mather Sets for Twist Maps and Geodesics on Tori},
Dynamics Reported \textbf{1} (1988), 1-56.
\bibitem{B1994}
V.~Bangert, \textit{Geodesic rays, Busemann functions and monotone twist maps},
Calc. Var. Partial Differential Equations 2, \textbf{1} (1994), 49–63.
\bibitem{BP}
M.~L.~Bialy and L.~V.~Polterovich, \textit{Geodesic Flows on the
two-dimensional Torus and Phase Transitions
'Commensurability-Noncommensurability'}, English translation:
Functional Anal. Appl. \textbf{20} (1986) no. 4, 260-266.
\bibitem{BR}
S.~V.~Bolotin and P.~H.~Rabinowitz, \textit{Some geometrical conditions for the
existence of chaotic geodesics on a torus}, Ergod. Th. and Dynam. Sys.
\textbf{22} (2002), 1407-1428.
\bibitem{BF} A.~V. ~ Bolsinov, A.~T. ~ Fomenko \textit{Integrable Geodesic Flows on
two-dimensional surfaces}, Monographs in Contemporary mathematics,
(1999).
\bibitem{BO}
R.~E.~Bowen, \textit{Entropy for Group Endomorphisms and
Homogeneous spaces}, Trans. of Am. Math. Soc. \textbf{153} (1971),
401-414.
\bibitem{DM}
J.~Denvir and R.~S.~MacKay, \textit{Consequences of contractible
geodesics on surfaces}, Trans. Amer. Math. Soc. \textbf{350}
(1998), no. 11, 4553-4568.
\bibitem{G}
E.~Glasmachers, \textit{Characterization of Riemannian metrics on $T^2$ with
and without positive topological entropy}, Thesis 2007.\\ \texttt{http://www-brs.ub.rub.de/netahtml/HSS/Diss/GlasmachersEva/}
\bibitem{GK}
E.~Glasmachers, G.~Knieper, \textit{Characterization of Geodesic Flows on $T^2$ with and
without Positive Topological Entropy}, GAFA, Geom.\ Funct.\ Anal. \textbf{20} (2010) no. 5, 1259-1277.
\bibitem{GO}
M.~A.~Grayson, \textit{Shortening embedded curves}, Ann. of Math. (2)
\textbf{129} (1989), no. 1, 71-111.
\bibitem{Hedlund}
G.~A.~Hedlund, \textit{Geodesics on a two-dimensional Riemannian manifold
with periodic coefficients}, Ann. of Math. \textbf{33} (1932), 719-739.
\bibitem{H}
E.~Hopf, \textit{Closed surfaces without conjugate points}, Proc. Natl.
Acad. Sci. \textbf{34} (1948), 47-51.
\bibitem{KH}
A.~Katok, B.~Hasselblatt, \textit{Introduction to the Modern Theory of
Dynamical Systems}, Cambridge University Press (1995).
\bibitem{K80}
A.~Katok, \textit{Lyapunov Exponents, Entropy and Periodic Orbits for
Diffeomorphisms}, Inst. Hautes \'{E}tudes Sci. Publ. Math. \textbf{51} (1980), 137-173.
\bibitem{Mather}
J.~N.~Mather, \textit{Existence of quasi-periodic orbits for twist
homeomorphisms of the annulus}, Topology, \textbf{21} (1982), 457-67.
\bibitem{Morse}
H.~M.~Morse, \textit{A fundamental class of geodesics on any closed surface of
genus greater than one}, Trans. Amer. Math. Soc. \textbf{26} (1924), 25-60.
\bibitem{Polt}
I.~V.~Polterovich, \textit{On a characterization of flat metrics on 2-torus},
J. of Dynamical and Control Systems \textbf{2} (1996), 89-101.
\bibitem{Sch}
J. ~P. ~Schr\"oder, \textit{Beschreibung der Geod\"atischen auf einem Torus mit Liouvillemetrik},
Bachelorarbeit, Ruhr-University Bochum, 2011.
\bibitem{Wo}
P.~Walters, \textit{An Introduction to Ergodic Theory}, Graduate Texts in
Mathematics, Springer-Verlag, New York, Berlin, Heidelberg (1982).
\end{thebibliography}
\end{document}